\documentclass[11pt]{amsart}
\usepackage[active]{srcltx}
\usepackage[utf8]{inputenc}
\usepackage{amsmath,mathtools}
\usepackage{stmaryrd}
\usepackage{MnSymbol}
\usepackage{hyperref}
\usepackage{tikz}
\usetikzlibrary{arrows,automata}
\usepackage{enumitem}
\usepackage{mathrsfs}
\usepackage{pst-plot}
\usepackage{auto-pst-pdf}
\usepackage[all]{xy}
\usepackage{stackrel}
\usepackage{listings}
\lstdefinelanguage{GAP}{%
  morekeywords={%
    Assert,Info,IsBound,QUIT,%
    TryNextMethod,Unbind,and,break,%
    continue,do,elif,%
    else,end,false,fi,for,%
    function,if,in,local,%
    mod,not,od,or,%
    quit,rec,repeat,return,%
    then,true,until,while%
  },%
  sensitive,%
  morecomment=[l]\#,%
  morestring=[b]",%
  morestring=[b]',%
}[keywords,comments,strings]
\usepackage[T1]{fontenc}
\usepackage{xcolor}
\usepackage{enumitem}
\newtheorem{thm}{Theorem}[section]
\newtheorem{cor}[thm]{Corollary}
\newtheorem{lem}[thm]{Lemma}
\newtheorem{prop}[thm]{Proposition}
\newtheorem{example}[thm]{Example}

\newcommand{\abs}[1]{\left\vert#1\right\vert}

\newcommand{\cc}[1]{\textcolor{red}{{#1}}}

\def\pv#1{\ensuremath{\mathsf{#1}}}
\newcommand{\J}{\mathrel{\mathscr J}} % J - relation
\newcommand{\R}{\mathrel{\mathscr R}} % R - relation
\newcommand{\eL}{\mathrel{\mathscr L}} % L - relation
\newcommand{\HH}{\mathrel{\mathscr H}}

\usepackage{caption}

\makeatletter
\newcommand{\mylabel}[2]{#2\def\@currentlabel{#2}\label{#1}}
\makeatother
\begin{document}
\title[The determinant of semigroups of the pseudovariety $\pv{ECOM}$]{The determinant of semigroups of the pseudovariety $\ensuremath{\mathsf{ECOM}}$}
\author{M.H. Shahzamanian}
\address{M.H. Shahzamanian\\ CMUP, Departamento de Matemática, Faculdade de Ciências, Universidade do Porto, Rua do Campo Alegre s/n, 4169--007 Porto (Portugal).}
\email{m.h.shahzamanian@fc.up.pt}
\thanks{ 2010 Mathematics Subject Classification. Primary 20M25, 16L60, 16S36.\\
Keywords and phrases: Frobenius algebra, semigroup determinant, paratrophic determinant, semigroup algebra.}

\begin{abstract}
The purpose of this paper is to compute the non-zero semigroup determinant of the class of finite semigroups in which every two idempotents commute.
This class strictly contains the class of finite semigroups that have central idempotents and the class of finite inverse semigroups.
This computation holds significance in the context of the extension of the MacWilliams theorem for codes over semigroup algebras.
\end{abstract}
\maketitle

%%%%%%%%%%%%%%%%%%%%%%%%%%%%%%%%%%%%%%%%%%%%%%%%%%%%%%%%%%%%%%%%%%%%%%%%%%%%%%%%%%%%%%%%%%%%%%%%%%%%%%%%%%%%%%%%%%%%%%%%%%%%%%%%%%%%%%%%%%%%%%
%%%%%%%%%%%%%%%%%%%%%%%%%%%%%%%%%%%%%%%%%%%%%%%%%%%%%%%%%%%%%%%%%%%%%%%%%%%%%%%%%%%%%%%%%%%%%%%%%%%%%%%%%%%%%%%%%%%%%%%%%%%%%%%%%%%%%%%%%%%%%%

\section{Introduction}

In the 1880s, Dedekind introduced the concept of the group determinant of finite groups and with Frobenius, began to study it in depth. At the same time, Smith also examined this concept, but in a different way, as outlined in \cite{Smith}. This study involved the investigation of the determinant of a $G \times G$ matrix, where the entry at the $(g, h)$ position is $x_{gh}$, with $G$ being a finite group and the $x_k$ are variables, for all
$k$ in $G$. Additionally, the study has been expanded to include finite semigroups with various research objectives \cite{Wilf, Lindstr, Wood}.
An application of the semigroup determinant for finite semigroups is the extension of the MacWilliams theorem for codes over a finite field to chain rings. 
Linear codes over a finite Frobenius ring have the extension property (see \cite{Wood-Duality}).
The non-zero semigroup determinant is an essential component in this application.
It is only non-zero when $\mathbb{C}S$ is a Frobenius algebra, which also means that it is unital.
This fact is demonstrated by Theorem 2.1 in \cite{Ste-Fac-det} or Proposition 18 in Chapter 16 of \cite{Okn}. 

In the paper~\cite{Ste-Fac-det} by Steinberg, he provides a factorization of the semigroup determinant of commutative semigroups.
The semigroup determinant is either zero or it factors into linear polynomials. Steinberg describes the factors and their multiplicities explicitly.
This work was a continuation of previous studies on commutative semigroups with Frobenius semigroup algebras by Ponizovski\u{\i}~\cite{Ponizovski} and Wenger~\cite{Wenger}.
Steinberg also showed that the semigroup determinant of an inverse semigroup can be computed as the semigroup determinant of a finite groupoid. 

The aim of the current paper is to understand under what conditions the determinant of a semigroup lying in the pseudovariety $\pv{ECOM}$ is non-zero and to study its factorization.
Essentially, it wishes to develop the ideas in Steinberg’s paper~\cite{Ste-Fac-det}. 
In~\cite{Ste-Fac-det}, the determinant of semigroups with central idempotents has been examined for the purpose of providing a factorization of commutative semigroups. %Also in~\cite{Ste-Fac-det}, the non-zero determinant can be recognized through Frobenius algebra.
The pseudovariety $\pv{ECOM}$ is, by a celebrated result of Ash, precisely the pseudovariety generated by finite inverse semigroups.
This is a larger class than that of the semigroups with central idempotents and also of inverse semigroups discussed in~\cite{Ste-Fac-det}. 
Then, this fact makes it a natural object of study.

We defines a partial order relation for this class of finite semigroups. This relation extends the natural partial ordering of the idempotents within the semigroup. This partial order relation is crucial for examining the determinant of these semigroups.
We then identify semigroups in this class with a non-zero determinant, studying their factorizations.
Our identification is more specific for this class of semigroups.

The paper is organized as follows. We begin with a preliminary section on semigroups and determinant of a semigroup. 
We present a partial order relation on the semigroups of the pseudovariety $\pv{ECOM}$ and investigate their properties.
We then proceed to compute the determinant of finite semigroups within this pseudovariety. 
To demonstrate the method, several examples are provided, and their calculations are performed using programs developed in C{\fontseries{b}\selectfont\#}. These examples are discussed in an appendix at the end of the paper.

%%%%%%%%%%%%%%%%%%%%%%%%%%%%%%%%%%%%%%%%%%%%%%%%%%%%%%%%%%%%%%%%%%%%%%%%%%%%%%%%%%%%%%%%%%%%%%%%%%%%%%%%%%%%%%%%%%%%%%%%%%%%%%%%%%%%%%%%%%%%%%
%%%%%%%%%%%%%%%%%%%%%%%%%%%%%%%%%%%%%%%%%%%%%%%%%%%%%%%%%%%%%%%%%%%%%%%%%%%%%%%%%%%%%%%%%%%%%%%%%%%%%%%%%%%%%%%%%%%%%%%%%%%%%%%%%%%%%%%%%%%%%%

\section{Preliminaries}

\subsection{Semigroups}

For standard notation and terminology relating to semigroups, we refer the reader to~\cite[Chap. 5]{Alm}, \cite[Chaps. 1-3]{Cli-Pre} and~\cite[Appendix A]{Rho-Ste}.

Let $S$ a finite semigroup. 
Let $a,b\in S$. 
We say that $a\R b$ if $aS^{1} = bS^{1}$, $a\eL b$ if $S^{1}a = S^{1}b$ and $a\HH b$ if $a\R b$ and $a\eL b$. 
Also, we say that $a\J b$, if $S^{1}aS^{1} = S^{1}bS^{1}$. 
Please observe that the symbol 1 in notation $S^1$ does not denote any specific element of $S$. 
If $S$ possesses an identity element, then $S^1= S$. However, if $S$ lacks an identity element, $S^1 =S \cup 1$, forming a semigroup with 1 as its identity element.
The relations $\R,\eL$, $\HH$ and $\J$ are Green’s relations, named after Green~\cite{Gre}.
We call $R_a,L_a,H_a$ and $J_a$, respectively, the $\R,\eL,\HH$ and $\J$-class containing $a$.
Also, we have $a \widetilde{\eL} b$ if and only if $a$ and $b$ have the same set of idempotent right identities, 
that is, $ae = a$ if and only if $be = b$ in the sense Fountain et al. \cite{Fou-Gom-Gou}. The relation $\widetilde{\R}$ is defined dually, 
and $\widetilde{\HH}=\widetilde{\eL}\wedge\widetilde{\R}$.
We write $\widetilde{L}_s$, $\widetilde{R}_s$ and $\widetilde{H}_s$ for the equivalence classes of $s$ of these relations, respectively.
For further results regarding this object see \cite{Lawson-Sem-Cat}.

An element $e$ of $S$ is called \emph{idempotent} if $e^2 = e$. 
The set of all idempotents of $S$ is denoted by $E(S)$. %more generally, for any $X\subseteq M$, we put $E(X)=X\cap E(M)$.
An idempotent $e$ of $S$ is the identity of the monoid $eSe$. 
The group of units $G_e$ of $eSe$ is called the maximal subgroup of $S$ at $e$. %Note that $G_e=H_e$.

A \emph{left ideal} of a semigroup $S$ is a non-empty subset $A$ of $S$ such that $SA \subseteq A$. 
A \emph{right ideal} of $S$ is defined dually, with the condition $AS \subseteq A$. 
An \emph{ideal} of $S$ is a subset of $S$ that is both a left and a right ideal.
%An element $s$ of $S$ is called (von Neumann) \emph{regular} if there exists an element $t\in S$ such that $sts=s$. 
%Note that an element $s$ is regular if and only if $s\eL e$, for some $e\in E(S)$, if and only if $s\R f$, for some $f\in E(S)$. 
%The semigroup $S$ is \emph{regular} if every element in $S$ is regular.
%A $\J$-class $J$ is regular if all its elements are regular, if and only if $J$ has an idempotent, if and only if $J^2\cap J\neq\emptyset$. 
The semigroup $S$ is \emph{inverse} if, for all $s \in S$, there is a unique element $s^{-1} \in S$ such that $ss^{-1}s =s$ and $s^{-1}ss^{-1}=s^{-1}$.
For an elements $s\in S$, $s^{\omega}$ is the limit of the sequence $(s^{n!})_n$. 

A \emph{pseudovariety} of semigroups is a class of finite semigroups that is closed under taking subsemigroups, homomorphic images, and finite direct products. The pseudovariety $\pv{S}$ consists of all finite semigroups, while the pseudovariety $\pv{G}$ is the class of all finite groups, $\pv{Sl}$ and $\pv{Com}$ are the pseudovarieties of all finite, respectively, semilattices and commutative semigroups.
The operator $\pv{E}$ associates a pseudovariety $\pv{V}$ to the class of finite semigroups such that the subsemigroup generated by the idempotents of the semigroup belongs to $\pv{V}$, which can be written as
$$\mathsf{EV} = \{S \in \mathsf{S} \mid \langle E(S)\rangle \in \mathsf{V}\}.$$
If a finite semigroup $S$ is a member of $\mathsf{ECom}$, then the subsemigroup generated by the idempotents of $S$ is equal to the set of idempotents of $S$. Therefore, the pseudovariety $\pv{ESl}$ is equal to the pseudovariety $\pv{ECom}$. By a celebrated result of Ash~\cite{Ash}, the pseudovariety generated by finite inverse semigroups is precisely the pseudovariety $\mathsf{ECom}$.
%the pseudovariety of finite semigroups whose idempotents commute.

\subsection{Incidence Algebras and M{\"o}bius Functions}

Let $(P,\leq)$ be a finite partially ordered set (poset). 
The \emph{incidence algebra} of $P$ over $\mathbb{C}$, which we denote $\mathbb{C}\llbracket P\rrbracket$, is the algebra of all functions $f\colon P \times P \rightarrow \mathbb{C}$ such that
$$f(x, y) \neq 0 \Rightarrow x \leq y$$
equipped with the convolution product
$$(f \ast g)(x, y) = \sum\limits_{x\leq z\leq y}f(x, z)g(z, y).$$
The \emph{convolution identity} is the delta function $\delta$ given by
\begin{equation*}
\delta(x,y)=
\begin{cases}
1 & \text{if}\ x=y\\
0 & \text{otherwise.}
\end{cases}
\end{equation*}
The \emph{zeta function}, denoted as $\zeta_P$, of the poset $P$ is an element of $\mathbb{C}\llbracket P\rrbracket$ 
given by
\begin{equation*}
\zeta_P(x,y)=
\begin{cases}
1 & \text{if}\ x \leq y\\
0 & \text{otherwise.}
\end{cases}
\end{equation*}
The function $\zeta_P$ is upper triangular with ones on the diagonal with respect to any linear order extending $P$. % and hence is unimodular. 
Therefore, $\zeta_P$ has an inverse over the integers called the \emph{M{\"o}bius function}, represented by $\mu_P$. 
In instances where the poset $P$ is clear from context, the subscript $P$ will be omitted.

Let $f$ be a function from $P$ to $\mathbb{C}$. 
By Applying M\"{o}bius inversion, 
if $g$ is the function from $P$ to $\mathbb{C}$ given by $g(x) = \sum\limits_{y\leq x}f(y)$ then $f(x)= \sum\limits_{y \leq x} \mu_P(y,x)g(y)$, for every $x \in P$.

We recommend that the reader refer to \cite{IncidenceAlgebras} for further information on this section.

\subsection{Determinant of a semigroup}

For standard notation and terminology relating to finite dimensional algebras, the reader is referred to \cite{Assem-Ibrahim} and \cite{Benson}.

A based algebra is a finite dimensional complex algebra $A$ with a distinguished basis $B$.
We often refer to the pair of the algebra and its basis as $(A, B)$.
The multiplication in the algebra is determined by its structure constants with respect to the basis $B$ defined by the equations
$$bb'\ =\sum\limits_{b''\in B}c_{b'',b,b'}b''$$
where $b, b'\in B$ and $c_{b'',b,b'}\in \mathbb{C}$.
Let $X_B=\{x_b\mid b \in B\}$ be a set of variables in bijection with $B$. 
These structure constants can be represented in a matrix called the Cayley table, which is a $B \times B$ matrix with elements from the polynomial ring $\mathbb{C}[X_B]$.
It is defined as a $B\times B$ matrix over $\mathbb{C}[X_B]$ with entries given by
$$C(A,B)_{b,b'} =\sum\limits_{b''\in B}c_{b'',b,b'}x_{b''}.$$
The determinant of this matrix, denoted by $\theta_{(A,B)}(X_B)$, is either identically zero or a homogeneous polynomial of degree $\abs{B}$.

Let $S$ be a finite semigroup.
The semigroup $\mathbb{C}$-algebra $\mathbb{C}S$ consists of all the formal sums $\sum\limits_{s\in S} \lambda_s s$, where $\lambda_s \in \mathbb{C}$ and
$s \in S$, with the multiplication defined by the formula
$$(\sum\limits_{s\in S} \lambda_s s)\cdot (\sum\limits_{t\in S} \mu_t t) =\sum\limits_{u=st\in S} \lambda_s\mu_t u.$$
Note that $\mathbb{C}S$ is a finite dimensional $\mathbb{C}$-algebra with basis $S$.
If $A =\mathbb{C}S$ and $B=S$, then the Cayley table of $C(S)=C(\mathbb{C}S,S)$ is the $S\times S$ matrix over $\mathbb{C}[X_S]$ with
$C(S)_{s,s'} =x_{ss'}$ where $X_S=\{x_s\mid s \in S\}$ is a set of variables in bijection with $S$.
We denote the determinant $C(\mathbb{C}S,S)$ by $\theta_S(X_S)$
and call it the (Dedekind-Frobenius) \emph{semigroup determinant} of $S$.
If the semigroup $S$ is fixed, we often write $X$ instead of $X_S$.
For more information on this topic, the reader is referred to \cite{Frobenius1903theorie}, \cite[Chapter16]{Okn} and \cite{Ste-Fac-det}.

The \emph{contracted semigroup algebra} of a semigroup $S$ with a zero element 0 on the complex numbers is defined as $\mathbb{C}_0S=\mathbb{C}S/\mathbb{C}0$; note that $\mathbb{C}0$ is a one-dimensional two-sided ideal. 
This algebra can be thought of as having a basis consisting of the non-zero elements of $S$ and having multiplication that extends that of $S$, but with the zero of the semigroup being identified with the zero of the algebra.
The contracted semigroup determinant of $S$, denoted by $\widetilde{\theta}_S$, is the determinant of $\widetilde{C}(S) =C(\mathbb{C}_0S, S\setminus\{0\})$, where $\widetilde{C}(S)_{s,t}$ is equal to $x_{st}$ if $st\neq 0$ and 0 otherwise. Let $\widetilde{X}=X_{S\setminus \{0\}}$ if $S$ is understood.

Steinberg in~\cite{Ste-Fac-det} provides a theorem regarding the behavior of the paratrophic determinant under isomorphism (or change of basis). 
We mention it as follows.

\begin{thm}\label{ThetaCal}
Let $(A, B)$ and $(A', B')$ be based algebras and let $f\colon A \rightarrow A'$ be a $\mathbb{C}$-algebra homomorphism. 
Let $P$ be the $B'\times B$ matrix of $f$ with respect to the bases $B$ and $B'$. 
Let $\widetilde{f}\colon \mathbb{C}[X_B] \rightarrow \mathbb{C}[X_{B'}]$ be the homomorphism $x_b \mapsto \sum\limits_{b' \in B'}P_{b',b}x_{b'}$ induced by $f$ (note that $\widetilde{f}(x_b)$ is a linear homogeneous polynomial). 
Then $\widetilde{f}(C(A, B)) =P^TC(A', B')P$. 
Therefore, if $f$ is an isomorphism, $\theta_{(A,B)}=\det(P)^2\widetilde{f}^{-1}(\theta_{(A',B')})$.
\end{thm}

According to Proposition 2.7 in~\cite{Ste-Fac-det} (the idea mentioned in~\cite{Wood}), there is a connection between the contracted semigroup determinant and the semigroup determinant of a semigroup $S$ with a zero element.
There is a $\mathbb{C}$-algebra isomorphism 
between the $\mathbb{C}$-algebra $\mathbb{C}S$ and the product algebra $\mathbb{C}_0S\times \mathbb{C}0$, which sends $s \in S$ to $(s,0)$.
Put $y_s=x_s-x_0$ for $s \neq 0$ and let $Y=\{y_s\mid s \in S\setminus \{0\}\}$. 
Then $\theta_S(X) =x_0\widetilde{\theta}_S(Y)$. 
Therefore, $\widetilde{\theta}_S(\widetilde{X})$ can be obtained from $\theta_S(X)/ x_0$ by replacing $x_0$ with 0.

%%%%%%%%%%%%%%%%%%%%%%%%%%%%%%%%%%%%%%%%%%%%%%%%%%%%%%%%%%%%%%%%%%%%%%%%%%%%%%%%%%%%%%%%%%%%%%%%%%%%%%%%%%%%%%%%%%%%%%%%%%%%%%%%%%%%%%%%%%%%%%
%%%%%%%%%%%%%%%%%%%%%%%%%%%%%%%%%%%%%%%%%%%%%%%%%%%%%%%%%%%%%%%%%%%%%%%%%%%%%%%%%%%%%%%%%%%%%%%%%%%%%%%%%%%%%%%%%%%%%%%%%%%%%%%%%%%%%%%%%%%%%%

\section{The relation $\ll$}\label{Relationll}
There are necessary conditions are given for a semigroup $S$ to have a non-zero $\theta_S(X)$, as stated in \cite{Ste-Fac-det}.
If $\theta_S(X)$ is not equal to 0, then the semigroup algebra $\mathbb{C}S$ is a unital algebra, according to Theorem 2.1 in~\cite{Ste-Fac-det}.
Furthermore, if $\theta_S(X)$ is not equal to 0, then $S$ satisfies the condition described in ~\ref{itm:star} below, as stated in Corollary 2.4 in~\cite{Ste-Fac-det}:
\begin{enumerate}
\item[\mylabel{itm:star}{($\star$)}] For each element $s$ in $S$, the number of elements of $S$ fixed by $s$ under left and right multiplication is the same.
\end{enumerate}

Let $S$ be a semigroup.
We define the functions $\varphi^{\ast}$ and $\varphi^{+}$ from $S$ to the power set of $E(S)$ % denoted by $\mathcal{P}(E(S))$, 
as follows:
$$\varphi^{\ast}(s)=\{e\in E(S)\mid se = s\}\ \text{and} \ \varphi^{+}(s)=\{e\in E(S)\mid es = s\}.$$

\begin{lem}\label{phiESNeqEmptyset}
If $S$ is finite and $\mathbb{C}S$ is a unital algebra, then the subsets $\varphi^{\ast}(s)$ and $\varphi^{+}(s)$ are non-empty, for every $s\in S$.
\end{lem}

\begin{proof}
Let $S=\{s_1,\ldots,s_n\}$ and $s\in S$.
Let $1_S=c_1s_1+\cdots+c_ns_n$, for some $c_1,\ldots,c_n\in \mathbb{C}$, be the identity of the algebra $\mathbb{C}S$.
Since $s1_S=s$, there exists an integer $1\leq i\leq n$ such that $c_i\neq 0$ and $ss_i=s$. 
Then, we have $ss_i^{\omega}=s$.
It follows that $s_i^{\omega}\in\varphi^{\ast}(s)$ and, thus, we get that $\varphi^{\ast}(s)\neq\emptyset$

Similarly, we have $\varphi^{+}(s)\neq\emptyset$.
\end{proof}

For the rest of this section, we will assume that $S$ belongs to the pseudovariety $\mathsf{ECom}$ and $S$ satisfies the following conditions designated as~\mylabel{itm:star2}{($\star\star$)}:
\begin{enumerate}
\item $S$ satisfies Condition~\ref{itm:star};
\item the semigroup algebra $\mathbb{C}S$ is a unital algebra.
\end{enumerate}

The elements of $E(S)$ commute, so we have $E(S)=\langle E(S)\rangle$. 
Hence, $E(S)$ is a meet semilattice with respect to the ordering $e \leq f$, if $ef = e$.
By Lemma~\ref{phiESNeqEmptyset}, the subsets $\varphi^{\ast}(s)$ and $\varphi^{+}(s)$ are non-empty.
Since $E(S)$ is a meet semilattice, $\varphi^{\ast}(s)$ and $\varphi^{+}(s)$ have unique minimum elements respect to the ordering $\leq$ which are denoted by $s^{\ast}$ and $s^{+}$, respectively.

Note that, $s^{\ast} = t^{\ast}$ if and only if $\varphi^{\ast}(s)=\varphi^{\ast}(t)$, and $s^{+} = t^{+}$ if and only if $\varphi^{+}(s)=\varphi^{+}(t)$. Then, the equivalence relations $\widetilde{\eL}$, $\widetilde{\R}$
and $\widetilde{\HH}$ can be described as follows:
\begin{enumerate}
\item $s \widetilde{\eL} t$ if $s^{\ast} = t^{\ast}$;
\item $s \widetilde{\R} t$ if $s^{+} = t^{+}$;
\item $s \widetilde{\HH} t$ if $s^{\ast} = t^{\ast}$ and $s^{+} = t^{+}$.
\end{enumerate}

Let $e$ be an idempotent of $\widetilde{L}_s$. We have $e = e^{\ast} = s^{\ast}$. Then $s^{\ast}$ is the unique idempotent in $\widetilde{L}_s$.
Also, $s^{+}$ is the unique idempotent in $\widetilde{R}_s$.
Note that, it is not necessarily true that for an idempotent $s\in S$, $s^{\ast}=s^{+}$ (as demonstrated in Example~\ref{Exa1}).

Let $s$ and $t$ be elements of $S$. Define $s \leq^{\ast} t$ if $s = ts^{\ast}$ and $s \leq^{+} t$ if $s = s^{+}t$.
Also, we define $s \ll t$, if $s = s^{+}ts^{\ast}$.
Since the elements within the set $E(S)$ commute, the relations $\leq^{\ast}$, $\leq^{+}$ and $\ll$ are partial orders extending the ordering $\leq$.
Moreover, the relations $\leq^{\ast}$ and $\leq^{+}$ are commute. 
%That means that $\leq^{\ast} \circ \leq^{+} = \leq^{+} \circ \leq^{\ast}$.
%$s \leq^{\ast} \circ \leq^{+} t$ means that there exists $u \in S$ such that $s\leq^{\ast} u \leq^ l t$. 

\begin{prop}\label{leqES}
The following conditions hold:
\begin{enumerate}
\item $s \leq^{\ast} t$ if and only if $s = te$ for some $e\in E(S)$, and $s \leq^{+} t$ if and only if $s = et$ for some $e \in E(S)$.
\item if $s \leq^{\ast} t$ or $s \leq^{+} t$, then $s^{+}\leq t^{+}$ and $s^{\ast}\leq t^{\ast}$. 
\item $s \ll t$ if and only if there exists an element $u$ in $S$
such that $s \leq^{\ast} u \leq^{+} t$ or $s \leq^{+} u \leq^{\ast} t$.
\item if $s \ll t$, then $s^{+}\leq t^{+}$ and $s^{\ast}\leq t^{\ast}$.
%\item the relations $\leq^{\ast}$, $\leq^{+}$ and $\ll$ are partial orders extending the ordering $\leq$.
%Moreover, the relations $\leq^{\ast}$ and $\leq^{+}$ are commute. 
\end{enumerate}
\end{prop}

\begin{proof}
(1) 
Suppose that there exists an idempotent $e \in E(S)$ such that $s = te$.
Hence, we have $se = s$ and, thus $s^{\ast} \leq e$. 
Therefore, we have $s = ss^{\ast} = tes^{\ast} = ts^{\ast}$. Hence, $s \leq^{\ast} t$. The converse is true as well, since $s^{\ast} \in E(S)$.

Similarly, the proposition holds for the relation $\leq^{+}$.

(2)
First, suppose that $s \leq^{\ast} t$.
Then, we have $s = ts^{\ast}$. Hence, if $e\in \varphi^{+}(t)$ then $e\in \varphi^{+}(s)$ and, thus, $s^{+}\leq t^{+}$.
Also, if $e\in \varphi^{\ast}(t)$ then $se = ts^{\ast}e = tes^{\ast} = ts^{\ast} = s$. Thus, $s^{\ast}\leq t^{\ast}$.

Similarly, if $s \leq^{+} t$, then $s^{+}\leq t^{+}$ and $s^{\ast}\leq t^{\ast}$.

(3)
First, suppose that there exists an elements $u$ in $S$ such that 
$s \leq^{\ast} u \leq^{+} t$ or $s \leq^{+} u \leq^{\ast} t$. By symmetry, we may assume that $s \leq^{\ast} u \leq^{+} t$.
Since $s \leq^{\ast} u$, we have $s = us^{\ast}$ and by part (2), we get that $s^{+}\leq u^{+}$.
As $u \leq^{+} t$, we have $u = u^{+}t$ and, thus, $s = u^{+}ts^{\ast}$. Therefore, it follows that $s = u^{+}ts^{\ast} = s^{+}u^{+}ts^{\ast} = s^{+}ts^{\ast}$. If follows that $s \ll t$.

Now, suppose that $s = s^{+}ts^{\ast}$. 
Hence, we have $s \leq^{\ast} s^{+}t \leq^{+} t$ and $s \leq^{+} ts^{\ast} \leq^{\ast} t$.
The result follows.

(4)
Since $s \ll t$, we have $s = s^{+}ts^{\ast}$. Now, as $t=tt^{\ast}$, we have $st^{\ast} = s^{+}ts^{\ast}t^{\ast} = s^{+}tt^{\ast}s^{\ast} = s^{+}ts^{\ast} = s$. Hence, $s^{\ast}\leq t^{\ast}$. 
Similarly, we have $s^{+}\leq t^{+}$.
%
%Since $s \leq^{\ast} s$ and $s \leq^{+} s$, we have $s \ll s$. Hence, the relation $\ll$ is reflexive.
%
%Suppose that $s \ll t$ and $t \ll w$, for some elements $s$, $t$ and $w$ in $S$.
%Then, by part (3), we have $s = s^{+}ts^{\ast}$ and $t = t^{+}wt^{\ast}$.
%By part (4), we have $s^{\ast}\leq t^{\ast}$ and $s^{+}\leq t^{+}$.
%It follows that $s = s^{+}t^{+}wt^{\ast}s^{\ast} = s^{+}ws^{\ast}$.
%Hence, we have $s \ll w$ and, thus, the relation $\ll$ is transitive. 
%
%Now, suppose that $s \ll t$ and $t \ll s$.
%Then, we have $s = s^{+}ts^{\ast}$ and $t = t^{+}st^{\ast}$.
%Hence, $s =  s^{+}ts^{\ast} =  s^{+}t^{+}st^{\ast}s^{\ast} = t^{+}s^{+}ss^{\ast}t^{\ast} = t^{+}st^{\ast} = t$.
%
%Therefore, $\ll$ is a partial order and it clearly extends the order $\leq$ on $E(S)$.
\end{proof}

If $s' \ll s$ and $t' \ll t$, it does not mean that $s't' \ll st$. The semigroup's determinant can be either zero or non-zero, as shown in  Examples~\ref{Exa2} and \ref{Exa3}.

For an idempotent $e \in E(S)$, we define the subsets $I^{\ast}_e=\{s \mid s^{\ast} < e \}$ and $I^{+}_e=\{s \mid s^{+} < e \}$. 
Note that if $s\in I^{\ast}_e$, then we have $e\in \varphi^{\ast}(s)$ and, if $s\in I^{+}_e$, then we have $e\in \varphi^{+}(s)$.
We have $Se=\{s\in S\mid se=s\}$ and $eS=\{s\in S\mid es=s\}$.
%Also, if $s\in Se$ and $s\not\in I^{\ast}_e$, then we have $s^{\ast}=e$. \\
The subsets $Se$ and $I^{\ast}_e$ are unions of $\widetilde{\eL}$-classes, with $Se \setminus \widetilde{L}_e = I^{\ast}_e$.
%Hence, we have $Se=I^{\ast}_e\bigcupdot \widetilde{L}_e$. 
%For every $f\in E(I^{\ast}_e)$, it is clear that $\widetilde{L}_f\subseteq I^{\ast}_e$.
%Then, we have $I^{\ast}_e=\bigcupdot\limits_{f\in E(I^{\ast}_e)}\widetilde{L}_f$.
Also, the subsets $eS$ and $I^{+}_e$ are unions of $\widetilde{\R}$-classes, with $eS \setminus \widetilde{R}_e = I^{+}_e$.\\
%we have $eS=I^{+}_e\bigcupdot \widetilde{R}_e$ and $I^{+}_e=\bigcupdot\limits_{f\in E(I^{\ast}_e)}\widetilde{R}_f$.

\begin{prop}\label{Ie}
The following properties hold:
\begin{enumerate}
\item the subsets $I^{\ast}_e$ and $I^{+}_e$ are left and right ideal of $S$, respectively.
\item we have $I^{\ast}_e=\emptyset$ if and only if $I^{+}_e=\emptyset$.
\item we have $I^{\ast}_e=\emptyset$ if and only if $Se =L_e=\widetilde{L}_e$ where $L_e$ is the $\eL$-class of the element $e$ in $S$.
\item we have $I^{+}_e=\emptyset$ if and only if $eS =R_e=\widetilde{R}_e$ where $R_e$ is the $\R$-class of the element $e$ in $S$.
\end{enumerate}
\end{prop}

\begin{proof}
(1)
Let $s \in I^{\ast}_e$ and $t \in S$. Then, we have $tss^{\ast} =ts$ and thus $(ts)^{\ast}\leq s^{\ast}<e$. 
Similarly, $I^{+}_e$ is an right ideal of $S$.

(2)
If $I^{\ast}_e\neq\emptyset$ then there exists an idempotent $f\in I^{\ast}_e$. Then, we have $f<e$. It follows that $I^{+}_e\neq\emptyset$.
Similarly, if $I^{+}_e\neq\emptyset$, then we have $I^{\ast}_e\neq\emptyset$.

(3)
If $Se\neq L_e$, then there exists an element $s\in S$ such that $se\not\in L_e$ and, thus, $(se)^{\omega}\neq e$. 
Hence, we have $(se)^{\omega}< e$. 
Since $((se)^{\omega})^{\ast}= (se)^{\omega}$, we have $(se)^{\omega}\in I^{\ast}_e$. 
Then, $I^{\ast}_e\neq \emptyset$.
Also, if $I^{\ast}_e\neq \emptyset$, then there exists an idempotent $f\in I^{\ast}_e$. 
So, we have $f < e$ and, thus, $fe=f$. 
Now, as $f\neq e$, we have $f\in Se\setminus L_e$. 
Then, $I^{\ast}_e= \emptyset$ if and only if, $Se =L_e$.

Now, we prove that $I^{\ast}_e=\emptyset$ if and only if $Se=\widetilde{L}_e$.
Suppose that $I^{\ast}_e\neq\emptyset$. Then, again there exists an idempotent $f$ in $I^{\ast}_e$ and, thus, we have $f<e$. 
Then $f\in Se\setminus\widetilde{L}_e$. 
Now, suppose that $Se\setminus\widetilde{L}_e \neq \emptyset$. 
Let $se\in Se\setminus\widetilde{L}_e$. 
Then, we have $(se)^{\ast}\neq e$ and, thus, we have $(se)^{\ast}< e$. Thus we have $I^{\ast}_e\neq\emptyset$.

(4) Similar to (3), the statement holds.
\end{proof}

\begin{prop}\label{Ie2}
We have $E(I^{\ast}_e)=E(I^{+}_e)$ and $\abs{I^{\ast}_e}=\abs{I^{+}_e}$. Moreover, we have $\abs{\widetilde{L}_e}=\abs{\widetilde{R}_e}$.
\end{prop}

\begin{proof}
Let $f\in E(I^{\ast}_e)$. 
Then, we have $f<e$. 
So $f\in E(I^{+}_e)$. 
Similarly, we have $E(I^{+}_e)\subseteq E(I^{\ast}_e)$.

Let $n$ be the largest integer such that there is a chain $e_1<e_2<\cdots<e_n=e$ with $e_1,\ldots,e_n\in E(I^{\ast}_e)$.
We show by induction on $n$ that $\abs{I^{\ast}_e}=\abs{I^{+}_e}$.
If $n=1$, then we have $I^{\ast}_e=\emptyset$ and by Proposition~\ref{Ie}.(2), $I^{\ast}_e=I^{+}_e$.
Assume then that $n > 1$. 
Let $f\in E(I^{\ast}_e)$.
By the assumption of the induction, we have $\abs{I^{\ast}_f}=\abs{I^{+}_f}$. 
Since $Sf=\{s\in S\mid sf=s\}$, $fS=\{s\in S\mid fs=s\}$ and $S$ satisfies Condition~\ref{itm:star}, we have $\abs{Sf}=\abs{fS}$. 
It follows that $\abs{\widetilde{L}_f}=\abs{\widetilde{R}_f}$. 
Since $I^{\ast}_e=\bigcup\limits_{f\in E(I^{\ast}_e)}\widetilde{L}_f$ and $I^{+}_e=\bigcup\limits_{f\in E(I^{\ast}_e)}\widetilde{R}_f$, we have $\abs{I^{\ast}_e}=\abs{I^{+}_e}$. 
Also, we have $\abs{\widetilde{L}_e}=\abs{\widetilde{R}_e}$.
\end{proof}

Note that, we can also conclude that $I^{\ast}_e=\emptyset$ if and only if $eSe =G_e=\widetilde{H}_e$ where $G_e$ is the maximal subgroup of $S$ at $e$ is the group of units of $eSe$. 

It is generally not true that for an idempotent $e\in E(S)$, the subset $I_e^{\ast}$ is an ideal (see Example~\ref{Exa4}).
Also, it is not necessarily true that for an idempotent $e\in E(S)$, the subset $I_e^{\ast}$ is a subset of $eSe$ (see Example~\ref{Exa5}).

%%%%%%%%%%%%%%%%%%%%%%%%%%%%%%%%%%%%%%%%%%%%%%%%%%%%%%%%%%%%%%%%%%%%%%%%%%%%%%%%%%%%%%%%%%%%%%%%%%%%%%%%%%%%%%%%%%%%%%%%%%%%%%%%%%%%%%%%%%%%%
%%%%%%%%%%%%%%%%%%%%%%%%%%%%%%%%%%%%%%%%%%%%%%%%%%%%%%%%%%%%%%%%%%%%%%%%%%%%%%%%%%%%%%%%%%%%%%%%%%%%%%%%%%%%%%%%%%%%%%%%%%%%%%%%%%%%%%%%%%%%%
%%%%%%%%%%%%%%%%%%%%%%%%%%%%%%%%%%%%%%%%%%%%%%%%%%%%%%%%%%%%%%%%%%%%%%%%%%%%%%%%%%%%%%%%%%%%%%%%%%%%%%%%%%%%%%%%%%%%%%%%%%%%%%%%%%%%%%%%%%%%%

\section{The determinant of finite semigroups of the pseudovariety $\pv{ECOM}$}

The relationships $\leq^{+}$, $\leq^{\ast}$ and $\ll$ are all coincide when $S$ has central idempotents. 
They can all be represented by $\leq$, where $s \leq t$ if there exists an idempotent $e \in E(S)$ such that $s = te$. 
Also, we have $I^{\ast}_e=I^{+}_e$, for an idempotent $e$. We can simplify the notation $I^{\ast}_e$ to $I_e$.
In this case, for an idempotent $e \in E(S)$, we define $\widetilde{H}^0_e$  as the Rees quotient $eSe/I_e$ if $I_e$ is non-empty.
By Proposition~\ref{Ie}, if $I_e$ is empty, then $\widetilde{H}^0_e=G_e\cup\{z\}$ where $z$ is an adjoined zero. 
Note that $\widetilde{H}^0_e$ forms a monoid with identity $e$ and can be identified with $\widetilde{H}_e\cup \{z\}$ where $z$ is a zero and, for $a,b \in\widetilde{H}_e$, we have that
\begin{equation*}
a\cdot b=
\begin{cases}
ab,&\text{if}\ ab\in \widetilde{H}_e\\
z,&\text{else}.
\end{cases}
\end{equation*}
The determinant of $S$ when $S$ has central idempotents, as stated in Theorem~6.6 of \cite{Ste-Fac-det}, is computed as follows:

\begin{thm}\label{central-idem}
Let $S$ be a finite idempotent semigroup has central idempotents. For $s \in S$, put $y_s=\sum\limits_{t\leq s}\mu_S(t, s)x_t$. Then
$\theta_S = \prod\limits_{e\in E(S)}\widetilde{\theta}_{\widetilde{H}^0_e}(Y_e)$ where $Y_e=\{y_s\mid s\in \widetilde{H}_e\}$.
\end{thm}

%The determinant of semigroups has central idempotents is examined in \cite{Ste-Fac-det}.
In this paper, our aim is to compute non-zero determinant of semigroups in $\pv{ECOM}$.
Using a similar method as in \cite{Ste-Fac-det} for semigroups has central idempotents. % to arrive at Theorem \ref{central-idem}, by utilizing the relation $\ll$. 
Then, let $S\in\mathsf{ECom}$ in this section such that $S$ satisfies Conditions~\ref{itm:star2}. 

Let $Z\colon \mathbb{C}S \rightarrow \mathbb{C}S$ be a map given by $Z(s) =\sum\limits_{s'\ll s}s'$ on $s \in S$, with a linear extension.
By applying M\"{o}bius inversion, we can establish an inverse for $Z$, making it bijective. As mentioned in the following proposition.

\begin{prop}\label{Z}
The mapping $Z$ is bijective. 
%In particular, $\mathbb{C}S$ is unital.
\end{prop}

For $s$ and $t$ in $S$,
one recursively defines two sequences $s_n$ and $t_n$ by
$$s_0=s, t_0=t$$
and
$$s_{n+1}=s_nt_n^{+}, t_{n+1}=s_n^{\ast}t_n.$$

\begin{prop}\label{somegaAndtomega}
We have $s_{n+1}=st_n^{+}$ and $t_{n+1}=s_n^{\ast}t$. Also, we have $s_{n+1}^{\ast}, t_{n+1}^{+}\leq s_n^{\ast}, t_n^{+}$.
\end{prop}

\begin{proof}
First, we prove that $s_{n+1}^{\ast}, t_{n+1}^{+}\leq s_n^{\ast}, t_n^{+}$.
Since
$$s_{n+1}s_n^{\ast} = s_nt_n^{+}s_n^{\ast} = s_ns_n^{\ast}t_n^{+} = s_nt_n^{+} = s_{n+1},$$
we have $s_{n+1}^{\ast}\leq s_n^{\ast}$.
Also, as
$$s_{n+1}t_n^{+} = s_nt_n^{+}t_n^{+} = s_nt_n^{+} = s_{n+1},$$
we have $s_{n+1}^{\ast}\leq t_n^{+}$. Similarly, we have $t_{n+1}^{+}\leq s_n^{\ast}, t_n^{+}$.

Let us show that $s_{n+1}=st_n^{+}$ and $t_{n+1}=s_n^{\ast}t$ by induction on $n$.
The case $n = 0$ is clear.
Assume then that $n > 0$, and that the result holds for smaller values of $n$.
We have $s_{n+1} = s_nt_n^{+}$. By hypothesis of induction, we have $s_n = st_{n-1}^{+}$. 
It follows that $s_{n+1} = st_{n-1}^{+}t_n^{+}$. Now, as $t_{n}^{+}\leq t_{n-1}^{+}$,
we have $s_{n+1} = st_n^{+}$. Similarly, we have $t_{n+1}=s_n^{\ast}t$.
\end{proof}

\begin{prop}\label{AandE}
We have $s_nt_n = st$, for every integer $n\geq 0$. Moreover, we have $s_net_n = set$, for every idempotent $e$.
\end{prop}

\begin{proof}
By Proposition~\ref{somegaAndtomega}, we have 
$s_n = st_{n-1}^{+}$ and 
$t_n = s_{n-1}^{\ast}t$.

Let us show that $st_m^{+}s_m^{\ast}t = st$ by induction on $m \leq n-1$.
First assume that $m = 0$.
Hence, we have 
$$st_0^{+}s_0^{\ast}t = 
  ss_0^{\ast}t_0^{+}t = st.$$
Assume then that $m > 0$, and that the result holds for smaller values of $m$.
By hypothesis of induction and Proposition~\ref{somegaAndtomega}, we have
\begin{align*}
      st_m^{+}s_m^{\ast}t
  & = st_{m-1}^{+}t_m^{+}s_m^{\ast}s_{m-1}^{\ast}t = s_mt_m^{+}s_m^{\ast}t_m \\
  & = s_ms_m^{\ast}t_m^{+}t_m = s_mt_m = st_{m-1}^{+}s_{m-1}^{\ast}t\\
  & = st.
\end{align*}  
Therefore, we have $s_nt_n = st_{n-1}^{+}s_{n-1}^{\ast}t = st$.

Similar as above, we have $s_net_n = set$.
\end{proof}

Since $S$ is finite, by Proposition~\ref{somegaAndtomega}, there exists an integer $m$ such that 
if $m\neq 0$, then $t_{m}^{+}=s_{m}^{\ast}$ and $t_{m-1}^{+} \neq s_{m-1}^{\ast}$, otherwise $t_0^{+}=s_0^{\ast}$. 
Then, we define a map \[\varepsilon\colon S\times S \rightarrow E(S)\]
given by 
\[\varepsilon(s,t)=t_{m}^{+} (=s_{m}^{\ast}).\]
%We denote the element $t_{m}^{+} (=s_{m}^{\ast})$ by $\varepsilon(s,t)$.
Then, we have \[s_{m+1}=s_{m}=s\varepsilon(s,t)\ \text{and}\ t_{m+1}=t_{m}=\varepsilon(s,t)t.\]
It follows that $(s\varepsilon(s,t))^{\ast}=(\varepsilon(s,t)t)^{+}=\varepsilon(s,t)$.

Let $E^{st}=\{e\in E(S)\mid set=st\}$. Since the idempotents $s^{\ast}$ and $t^{+}$ are in the set $E^{st}$, $E^{st}$ is non empty.
By Propositions~\ref{somegaAndtomega} and \ref{AandE}, we have $s\varepsilon(s,t)t = st$ and $\varepsilon(s,t)\leq t^{+}, s^{\ast}$.
It is possible that the subset $E^{st}$ has more than one minimal element (see Example~\ref{Exa8}). 

We define the multiplication
$\sharp \colon \mathbb{C}S \times \mathbb{C}S \rightarrow \mathbb{C}S$ 
as follows:\\
\begin{equation*}
s\sharp t = 
\begin{cases}
  st,& \text{if}\ s^{+}=(st)^{+}, t^{\ast}=(st)^{\ast}\ \text{and}\ s^{\ast}=t^{+};\\ %\text{and if}\ set=st,\\ &\text{for some idempotent}\ e,\ \text{then}\ s^{\ast}\leq e;\\
  0,& \text{otherwise},
\end{cases}
\end{equation*}
for every $s,t\in S$. 
The multiplication operation $\sharp$ is not always associative (as demonstrated in Example~\ref{Exa9}).
Also, we define the multiplication
$\sharp\limits^{e} \colon \mathbb{C}S \times \mathbb{C}S \rightarrow \mathbb{C}S$, for some $e\in E(S)$,
as follows:\\
\begin{equation*}
s\sharp^{e} t = 
\begin{cases}
  st & \text{if}\ s^{+}=(st)^{+}, t^{\ast}=(st)^{\ast}\ \text{and}\ s^{\ast}=t^{+}=e;\\
  0& \text{otherwise},
\end{cases}
\end{equation*}

According to Proposition~\ref{Z}, the mapping $Z$ is a one-to-one correspondence. 
We define the following multiplication on $\mathbb{C}S \times \mathbb{C}S$.

\begin{equation*}
Z(s)\boldsymbol{*}Z(t) = 
\sum\limits_{\substack{s'\ll s,\\ t'\ll t}}
s'\stackrel{\varepsilon({s'}^{+}s,t{t'}^{\ast})}{\sharp} t'
\end{equation*}
By applying M\"{o}bius inversion, we have 
\[
u= \displaystyle\sum\limits_{u'\ll u}\mu_S(u', u)Z(u')\
\text{and}\ 
v= \displaystyle\sum\limits_{v'\ll v}\mu_S(v', v)Z(v'),\]
 for every $u,v \in S$. 
Then \[u\boldsymbol{*}v=\displaystyle\sum\limits_{u'\ll u,v'\ll v}\mu_S(u', u)\mu_S(v', v)Z(u')\boldsymbol{*}Z(v').\]

\begin{prop}\label{s*t}
We have $$Z(s)\boldsymbol{*}Z(t)=Z(st)$$ %$$\sum_{s'\ll s}s'\boldsymbol{\cdot} \sum_{t'\ll t}t'= \sum_{u\ll st}u$$
for every $s,t\in S$.
\end{prop}

\begin{proof}
Let $s'\ll s$ and $t'\ll t$. 
The first step is to prove that $s't'\ll st$, if $s'\stackrel{\varepsilon({s'}^{+}s,t{t'}^{\ast})}{\sharp} t'\neq 0$.
It is clear that the following condition holds:
\begin{equation}\label{ConditionsOfs't'2}
{s'}^{+}=(s't')^{+}, {t'}^{\ast}=(s't')^{\ast}\ \text{and}\ {s'}^{\ast}={t'}^{+}.
\end{equation}
Since $s'\ll s$ and $t'\ll t$, we have $s' = {s'}^{+}s{s'}^{\ast}$ and $t' = {t'}^{+}t{t'}^{\ast}$.
Now, as $s'\stackrel{\varepsilon({s'}^{+}s,t{t'}^{\ast})}{\sharp} t'\neq 0$, we have ${s'}^{\ast}={t'}^{+}=\varepsilon({s'}^{+}s,t{t'}^{\ast})$.
Hence, we have $s't' = {s'}^{+}s{s'}^{\ast}{t'}^{+}t{t'}^{\ast} = {s'}^{+}s\varepsilon({s'}^{+}s,t{t'}^{\ast})t{t'}^{\ast} = {s'}^{+}st{t'}^{\ast}$.
As ${s'}^{+}=(s't')^{+}$ and ${t'}^{\ast}=(s't')^{\ast}$, we have $s't'\ll st$.

Secondly, we need to show that if $u\ll st$, for some $u$ in $S$, there exists a unique pair $(s^{\circ},t^{\circ})$ such that $s^{\circ}\ll s$, $t^{\circ}\ll t$,  $s^{\circ}\stackrel{\varepsilon({s^{\circ}}^{+}s,t{t^{\circ}}^{\ast})}{\sharp} t^{\circ}\neq 0$, and $s^{\circ}t^{\circ}=u$.
Since $u\ll st$, we have $u = u^{+}stu^{\ast}$.
Let $s' = u^{+}s$
and $t' = tu^{\ast}$.
It is easily follows that ${s'}^{+} = u^{+}$ and ${t'}^{\ast} = u^{\ast}$.
Let $s''=s'\varepsilon(s',t')$ and $t''=\varepsilon(s',t')t'$. 
Clearly, we have $s''\ll s$, $t''\ll t$ and ${s''}^{\ast} = {t''}^{+} = \varepsilon(s',t')$.
As $s'\varepsilon(s',t')t'=s't'=u$, we have ${s''}^{+} = u^{+}$ and ${t''}^{\ast} = u^{\ast}$. 
Then, we have \[s''\stackrel{\varepsilon({s''}^{+}s=s',t{t''}^{\ast}=t')}{\sharp}t''=s''t''=s'\varepsilon(s',t')t'=u.\]

There is then our desired pair. 
Now, we prove the uniqueness of this existence. 
Let $s_1,s_2,t_1,t_2\in S$ such that $s_1,s_2\ll s$, $t_1,t_2\ll t$, $s_1t_1=s_2t_2=u$, and 
$s_1\stackrel{\varepsilon({s_1}^{+}s,t{t_1}^{\ast})}{\sharp} t_1,s_2\stackrel{\varepsilon({s_2}^{+}s,t{t_2}^{\ast})}{\sharp} t_2\neq 0$.
Then, the pairs $(s_1,t_1)$ and $(s_2,t_2)$ satisfy Condition~(\ref{ConditionsOfs't'2}) and, thus, we have
 $s_1^{+}=s_2^{+}=u^{+}$ and $t_1^{\ast}=t_2^{\ast}=u^{\ast}$. 
Also, we have $s_1^{\ast}=s_2^{\ast}=t_1^{+}=t_2^{+}=\varepsilon(u^{+}s,tu^{\ast})$.
Therefore, we have $s_1=u^{+}s\varepsilon(u^{+}s,tu^{\ast})=s_2$ and $t_1=\varepsilon(u^{+}s,tu^{\ast})su^{\ast}=t_2$. 

The result follows.
\end{proof}

The following theorem is a generalization of the result by Steinberg in~\cite{Ste-Fac-det} for finite semigroups has central idempotents, and uses a similar construction. 

\begin{thm}\label{Zisom}
The mapping $Z$ is an isomorphism of $\mathbb{C}$-algebras. 
\end{thm}

\begin{proof}
By Propositions~\ref{Z} and~\ref{s*t}, the result follows.
\end{proof}

\begin{lem}\label{sPrimeAndsPrimePrime} %lll49 = DF21.EpsilonsprimeAndsFi(n, StBSS, ES, OrderES);
Suppose that \(s''\ll s' \ll s\) and \(t''\ll t' \ll t\).
If \(s''\stackrel{\varepsilon({s''}^{+}s,t{t''}^{\ast})}{\sharp} t''\neq 0\) then, we have \(\varepsilon({s''}^{+}s',t'{t''}^{\ast}) = \varepsilon({s''}^{+}s,t{t''}^{\ast})\).
\end{lem}

\begin{proof}
Let 
\(a = {s''}^{+}s\),
\(b = t{t''}^{\ast}\),
\(c = {s''}^{+}s'\),
and
\(d = t'{t''}^{\ast}\).

%====================================\\
We prove, by induction on $n$, that 
\begin{equation}\label{eq1}
c_n^{\ast} \leq a_n^{\ast}\ \text{and}\ d_n^{+} \leq b_n^{+},
\end{equation}
for every $0\leq n$.
We have
\begin{align*}
c_0a_0^{\ast} 
&= {s''}^{+}s'({s''}^{+}s)^{\ast} 
 = {s''}^{+}{s'}^{+}s{s'}^{\ast}({s''}^{+}s)^{\ast}
 = {s'}^{+}{s''}^{+}s({s''}^{+}s)^{\ast}{s'}^{\ast}
 = {s'}^{+}{s''}^{+}s{s'}^{\ast}\\
&= {s''}^{+}{s'}^{+}s{s'}^{\ast} = {s''}^{+}s' = c_0.%\ \text{and}\\                                               
%a_0c_0^{\ast}
%&= {s''}^{+}s({s''}^{+}s')^{\ast} 
% = {s''}^{+}{s'}^{+}s({s''}^{+}s')^{\ast}{s'}^{\ast}
% = {s''}^{+}{s'}^{+}s{s'}^{\ast}({s''}^{+}s')^{\ast}\\
%&= {s''}^{+}s'({s''}^{+}s')^{\ast}
% = {s''}^{+}s' = c_0.                                               
\end{align*}
It follows that \(c_0^{\ast} \leq a_0^{\ast}\).
Similarly, we have \(d_0^{+} \leq b_0^{+}\).
Assume then that $n \geq 1$. 
By the assumption of the induction, we have \(c_{n-1}^{\ast} \leq a_{n-1}^{\ast}\) and \(d_{n-1}^{+} \leq b_{n-1}^{+}\). 
%We have
%\begin{align*}
%c_1a_1^{\ast} &= c_0d_{0}^{+}(a_0b_{0}^{+})^{\ast} = c_0(a_0b_{0}^{+})^{\ast}d_{0}^{+}
%= {s''}^{+}s'({s''}^{+}s(t{t''}^{\ast})^{+})^{\ast}(t'{t''}^{\ast})^{+}\\
%&= {s''}^{+}{s'}^{+}s{s'}^{\ast}({s''}^{+}s(t{t''}^{\ast})^{+})^{\ast}(t'{t''}^{\ast})^{+}
%= {s'}^{+}{s''}^{+}s({s''}^{+}s(t{t''}^{\ast})^{+})^{\ast}{s'}^{\ast}(t'{t''}^{\ast})^{+}\\
%&= {s'}^{+}{s''}^{+}s({s''}^{+}s(t{t''}^{\ast})^{+})^{\ast}{s'}^{\ast}(t{t''}^{\ast})^{+}(t'{t''}^{\ast})^{+}\\
%&= {s'}^{+}{s''}^{+}s(t{t''}^{\ast})^{+}({s''}^{+}s(t{t''}^{\ast})^{+})^{\ast}{s'}^{\ast}(t'{t''}^{\ast})^{+}
%= {s'}^{+}{s''}^{+}s(t{t''}^{\ast})^{+}{s'}^{\ast}(t'{t''}^{\ast})^{+}\\
%&= {s'}^{+}{s''}^{+}s{s'}^{\ast}(t'{t''}^{\ast})^{+}
%= {s''}^{+}s'(t'{t''}^{\ast})^{+}=c_1
%\end{align*}
%====================================\\
Since
\begin{align*}
c_na_n^{\ast} 
&= c_0d_{n-1}^{+}(a_0b_{n-1}^{+})^{\ast}
 = c_0(a_0b_{n-1}^{+})^{\ast}d_{n-1}^{+} 
 = c_0(a_0b_{n-1}^{+})^{\ast}b_{n-1}^{+}d_{n-1}^{+}\\
&= {s''}^{+}s'(a_0b_{n-1}^{+})^{\ast}b_{n-1}^{+}d_{n-1}^{+} 
 = {s''}^{+}{s'}^{+}s{s'}^{\ast}(a_0b_{n-1}^{+})^{\ast}b_{n-1}^{+}d_{n-1}^{+}\\
&= {s'}^{+}{s''}^{+}sb_{n-1}^{+}(a_0b_{n-1}^{+})^{\ast}d_{n-1}^{+}{s'}^{\ast}
 = {s'}^{+}a_0b_{n-1}^{+}(a_0b_{n-1}^{+})^{\ast}d_{n-1}^{+}{s'}^{\ast}\\
&= {s'}^{+}a_0b_{n-1}^{+}d_{n-1}^{+}{s'}^{\ast}
 = {s'}^{+}a_0d_{n-1}^{+}{s'}^{\ast}
 = {s'}^{+}{s''}^{+}sd_{n-1}^{+}{s'}^{\ast}\\
&= {s''}^{+}{s'}^{+}s{s'}^{\ast}d_{n-1}^{+} = {s''}^{+}s'd_{n-1}^{+} = c_0d_{n-1}^{+}
 = c_n,
\end{align*}
we have \(c_n^{\ast} \leq a_n^{\ast}\).
Similarly, we have \(d_n^{+} \leq b_n^{+}\).

%=====================================================

Since \(s''\stackrel{\varepsilon({s''}^{+}s,t{t''}^{\ast})}{\sharp} t''\neq 0\), we have 
\begin{equation}\label{eq2}
{s''}^{\ast}={t''}^{+}=\varepsilon({s''}^{+}s,t{t''}^{\ast}).
\end{equation}

%=====================================================

Now, we prove, by induction on $n$, that 
\begin{equation}\label{eq3}
{s''}^{\ast} \leq c_{n}^{\ast},d_{n}^{+},
\end{equation}
for every $0\leq n$.
Since
\[
s'' c_0^{\ast} = {s''}^+s'{s''}^{\ast} ({s''}^{+}s')^{\ast} = {s''}^+s' ({s''}^{+}s')^{\ast}{s''}^{\ast}
 = {s''}^+s' {s''}^{\ast} = s'',
\]
we have \({s''}^{\ast} \leq c_0^{\ast}\). 
Similarly, we have \({s''}^{\ast} \leq d_0^{+}\).
%\[
%s'' c_1^{\ast} = {s''}^+s'{s''}^{\ast} ({s''}^{+}s'd_0^{+})^{\ast} = {s''}^+s'd_0^{+} ({s''}^{+}s'd_0^{+})^{\ast}{s''}^{\ast}
% = {s''}^+s' {s''}^{\ast}d_0^{+} = s''d_0^{+} = s''{s''}^{\ast}d_0^{+} = s''{s''}^{\ast}= s'',
%\]

Now, assume then that $n \geq 1$. 
By the assumption of the induction, we have \({s''}^{\ast} \leq c_{n-1}^{\ast}, d_{n-1}^{+}\). 
Since
\begin{align*}
s'' c_n^{\ast} 
&= s''(c_0d_{n-1}^{+})^{\ast} 
 = s''(c_0d_{n-1}^{+})^{\ast}d_{n-1}^{+} 
 = {s''}^+s'{s''}^{\ast}({s''}^{+}s'd_{n-1}^{+})^{\ast}d_{n-1}^{+}\\
&= {s''}^+s'd_{n-1}^{+}({s''}^{+}s'd_{n-1}^{+})^{\ast}{s''}^{\ast} 
 = {s''}^+s'd_{n-1}^{+}{s''}^{\ast} 
 = {s''}^+s'{s''}^{\ast}d_{n-1}^{+} 
 = s''d_{n-1}^{+}\\
&= s''s''^{\ast}d_{n-1}^{+} = s''s''^{\ast} = s'',
\end{align*}
we have \({s''}^{\ast} \leq c_{n}^{\ast}\). Similarly, we have \({s''}^{\ast} \leq d_{n}^{+}\).

By~(\ref{eq1}), (\ref{eq2}), (\ref{eq3}), there exists an integer \(m\) such that \(a_m=b_m=c_m=d_m={s''}^{\ast}\). The result follows.
\end{proof}

In Lemma~\ref{sPrimeAndsPrimePrime}, the converse may not hold that
\(s''\stackrel{\varepsilon({s''}^{+}s',t'{t''}^{\ast})}{\sharp} t''\neq 0\) and \(s''\stackrel{\varepsilon({s''}^{+}s,t{t''}^{\ast})}{\sharp} t''= 0\) (see Example~\ref{Exa10}).
Moreover, it may be the case 
\[{s''}^{\ast}={t''}^{+}=\varepsilon({s''}^{+}s',t'{t''}^{\ast})\] and \(s''\sharp t''= 0\), but \(\varepsilon({s''}^{+}s,t{t''}^{\ast}) \neq {s''}^{\ast}\) (see Example~\ref{Exa11}).

\begin{lem}\label{stEpsilon}
Let $s,t\in S$ and $e,f \in E(S)$ with \(t^{+} \leq e\) and \(s^{\ast} \leq f\).
We have \[\varepsilon(s,t) = \varepsilon(se,t) = \varepsilon(s,ft) = \varepsilon(se,ft).\]
\end{lem}

\begin{proof}
Let \(a_0 = s, b_0 =  t, c_0 = se\), and \(d_0 =  t\).
Since \[ed_n=ec_{n-1}^{\ast}t=c_{n-1}^{\ast}et=c_{n-1}^{\ast}et^{+}t=c_{n-1}^{\ast}t=d_n,\] we have $d^{+}_n \leq e$, for every $n \geq 0$. 
We show by induction on $n$ that $a_{2n+1}=c_{2n+1}$ and $b_{2n}=d_{2n}$, for every $n \geq 0$.
We have 
\(a_1 = st^{+}\)
and
\(c_1 = set^{+}\).
Since \(t^{+} \leq e\), we have \(a_1 = c_1\).
Assume then that $n > 0$, and that the result holds for smaller values of $2n+1$.
We have 
\(a_{2n+1} = sb_{2n}^{+}\)
and
\(c_{2n+1} = sed_{2n}^{+}\). Since $d^{+}_{2n} \leq e$ and \(b_{2n}^{+} = d_{2n}^{+}\), we have \(a_{2n+1} = c_{2n+1}\). Now, as \(a_{2n+1} = c_{2n+1}\), we have \(b_{2n+2} = a^{\ast}_{2n+1}t = c^{\ast}_{2n+1}t = d_{2n+2}\).

Therefore, we have \(\varepsilon(s,t) = \varepsilon(se,t)\). Similarly, we have \(\varepsilon(s,t) = \varepsilon(s,ft)\).

As \(s^{\ast} \leq f\), we get that \(sf=s\). 
Hence, we have \(sef=se\) and, thus, \((se)^{\ast} \leq f\). 
Now, as mentioned above, we have
\(\varepsilon(se,t) = \varepsilon(se,ft)\).

The result follows.
\end{proof}

By Lemma~\ref{stEpsilon}, the following corollary holds.

\begin{cor}\label{sPrimeAndsPrimePrime1}
Let \(s' \ll s\) and \(t' \ll t\) with \({s'}^{\ast} = {t'}^{+}\).
We have \[\varepsilon({s'}^{+}s,t') = \varepsilon(s',t{t'}^{\ast}) = {s'}^{\ast}.\]
\end{cor}

Let $s,t\in S$.
We have \[
s= \displaystyle\sum\limits_{s'\ll s}\mu_S(s', s)Z(s')\
\text{and}\ 
t= \displaystyle\sum\limits_{t'\ll t}\mu_S(t', t)Z(t').\]
Then, we get that
\begin{align*}
s\boldsymbol{*}t
&=
\displaystyle\sum\limits_{\substack{s'\ll s,\\ t'\ll t}}\mu_S(s', s)\mu_S(t', t)Z(s')\boldsymbol{*}Z(t')\\
&=
\displaystyle\sum\limits_{\substack{s'\ll s,\\ t'\ll t}}\mu_S(s', s)\mu_S(t', t)
\big(
\sum\limits_{\substack{s''\ll s',\\ t''\ll t'}}
(s''\stackrel{\varepsilon({s''}^{+}s',t'{t''}^{\ast})}{\sharp} t'')
\big)\\
&=
\sum\limits_{\substack{s''\ll s,\\ t''\ll t}}\big[\sum\limits_{\substack{s''\ll s'\ll s,\\ t''\ll t'\ll t,\\ 
s''\stackrel{\varepsilon({s''}^{+}s',t'{t''}^{\ast})}{\sharp} t''\neq 0}}
\big(
\mu_S(s', s)\mu_S(t', t)
\big)\big]
s''t''\\
&=
\sum\limits_{\substack{s''\ll s,\\ t''\ll t}}
\big[
\sum\limits_{t''\ll t'\ll t}
\big(
\sum\limits_{\substack{s''\ll s'\ll s,\\ s''\stackrel{\varepsilon({s''}^{+}s',t'{t''}^{\ast})}{\sharp} t''\neq 0}}
\mu_S(s', s)
\big)
\mu_S(t', t)
\big]
s''t''%\\
%&=
%\sum\limits_{\substack{s''\ll s,\\ t''\ll t}}
%\big[
%\sum\limits_{s''\ll s'\ll s}
%\big(
%\sum\limits_{\substack{t''\ll t'\ll t,\\ s''\stackrel{\varepsilon({s''}^{+}s',t'{t''}^{\ast})}{\sharp} t''\neq 0}}
%\mu_S(t', t)
%\big)
%\mu_S(s', s)
%\big]
%s''t''
.
\end{align*}
We define the function \(\xi\colon S^{(s)}\times S^{(t)}\rightarrow \mathbb{C}\), where \(S^{(s)}=\{s'' \in S \mid s'' \ll s\}\), for every \(s \in S\), as follows:
\begin{align*}
\xi(s'',t'') 
&= 
\sum\limits_{t''\ll t'\ll t}
\big(
\sum\limits_{\substack{s''\ll s'\ll s,\\ s''\stackrel{\varepsilon({s''}^{+}s',t'{t''}^{\ast})}{\sharp} t'' \neq 0}}
\mu_S(s', s)
\big)
\mu_S(t', t)\\
&=
\sum\limits_{s''\ll s'\ll s}
\big(
\sum\limits_{\substack{t''\ll t'\ll t,\\ s''\stackrel{\varepsilon({s''}^{+}s',t'{t''}^{\ast})}{\sharp} t'' \neq 0}}
\mu_S(t', t)
\big)
\mu_S(s', s),
\end{align*}
for every \(s'' \in S^{(s)}\) and \(t'' \in S^{(t)}\).

\begin{lem}\label{LemmaMain}
Let \(s'' \in S^{(s)}\) and \(t'' \in S^{(t)}\) with \(s'' \sharp t'' \neq 0\).
If \(s^{\ast} \neq t^{+}\)
then
\(\xi(s'',t'') = 0\). 
\end{lem}

\begin{proof}
Since \(s^{\ast} \neq t^{+}\), we have \(s^{\ast} \not\leq t^{+}\) or \(t^{+} \not\leq s^{\ast}\).
By symmetry, we may assume that \(s^{\ast} \not\leq t^{+}\).
Hence, we have \(st^{+} \neq s\).
Since \(s'' \sharp t'' \neq 0\), we get that \({s''}^{\ast} = {t''}^{+}\).
Now, as \({t''} \ll t\), we have \({s''}^{\ast} \leq t^{+}\) and, thus, we obtain \({s''} \ll st^{+} \ll s\).
Hence, we have \({s''} \neq s\).

Let \(t''\ll t'\ll t\). We prove that \(\big(
\sum\limits_{\substack{s''\ll s'\ll s,\\ s''\stackrel{\varepsilon({s''}^{+}s',t'{t''}^{\ast})}{\sharp} t'' \neq 0}}
\mu_S(s', s)
\big)
 = 0\), thereby completing the proof of the lemma.

If \(\varepsilon({s''}^{+}s',t'{t''}^{\ast}) = {s''}^{\ast}\), for all \(s''\ll s'\ll s\), and considering that \({s''} \neq s\), it follows that \(\sum\limits_{s''\ll s'\ll s}\mu_S(s', s) = 0\).

Next, suppose that there is an element \(s''\ll x\ll s\) such that \(\varepsilon({s''}^{+}x,t'{t''}^{\ast}) \neq {s''}^{\ast}\).
By Lemma~\ref{sPrimeAndsPrimePrime}, 
there exist pairwise distinct elements \(s''\ll s_1,\ldots, s_{n} \ll s\) 
such that
the following conditions hold:
\begin{enumerate}
\item \(\varepsilon({s''}^{+}s_i,t'{t''}^{\ast}) \neq {s''}^{\ast}\), for every \(1\leq i \leq n\);
\item if \(u \ll s_i\) with \(u \neq s_i\), then \(\varepsilon({s''}^{+}u,t'{t''}^{\ast}) = {s''}^{\ast}\), for all \(s''\ll u\ll s\);
\item if \(\varepsilon({s''}^{+}u,t'{t''}^{\ast}) \neq {s''}^{\ast}\), for some \(s''\ll u\ll s\), then there exists an integer \(i\) such that \(s_i \ll u\).
\end{enumerate}

Since \(t' \ll t\), we have \({t'}^{+} \leq t^{+}\). 
Then, by Lemma~\ref{stEpsilon}, we have 
\[
\varepsilon({s''}^{+}s_i,t'{t''}^{\ast}) 
= 
\varepsilon({s''}^{+}s_i,t^{+}t'{t''}^{\ast}) 
= 
\varepsilon({s''}^{+}s_it^{+},t'{t''}^{\ast}).\]
Now, as \(s_it^{+} \ll s_i\), applying the second condition mentioned above, we conclude that \(s_it^{+} = s_i\).
It follows that \(s_i \ll st^{+}\), for every \(1\leq i \leq n\).

Let \(X_i = \{s'' \ll x \ll s \mid s_i \ll x \}\), for every \(1\leq i \leq n\).
Consider integers \(i_1, \ldots, i_m\)
and elements \(x, y \in X_{i_1}\bigcap \cdots \bigcap X_{i_m} \).
Let \(u = x^{+}y^{+}sx^{\ast}y^{\ast}\). Since \(x,y \ll s\), we have \(u \ll x,y\).
Let \(1 \leq j \leq m\).
Since \(s_{i_j} \ll x, y\), we have \(s_{i_j}^{+} \leq x^{+}, y^{+}\) and \(s_{i_j}^{\ast} \leq x^{\ast}, y^{\ast}\).
Hence, we obtain that \(s_{i_j}^{+}x^{+}y^{+}sx^{\ast}y^{\ast}s_{i_j}^{\ast} = s_{i_j}^{+}ss_{i_j}^{\ast} = s_{i_j}\).
Then, we have \(s_{i_j} \ll u\).
It follows that the subset \(X_{i_1}\bigcap \cdots \bigcap X_{i_m}\) has a minimum element \(x_{i_1, \ldots,i_m}\) regarding the relation \(\ll\).
%If \(x_{i_1, \ldots,i_m} \ll z\ll s\), for some element \(z\), then we have \(s_{i_j} \ll z\), for every integer \(1 \leq j \leq m\). Hence, we have \(z \in X_{i_1}\bigcap \cdots \bigcap X_{i_m}\).
Now, as the distinct elements \(s\) and \(st^{+}\) are in the subset \(X_{i_1}\bigcap \cdots \bigcap X_{i_m}\),
we have 
\[
\sum\limits_{\substack{s''\ll x\ll s\\ x \in X_{i_1}\bigcap \cdots \bigcap X_{i_m}}}\mu_S(x, s) 
= 
\sum\limits_{x_{i_1, \ldots,i_m}\ll x\ll s}\mu_S(x, s)
=
0.
\]
Therefore, we conclude that
\(\big(\sum\limits_{\substack{s''\ll x\ll s\\ \varepsilon({s''}^{+}x,t'{t''}^{\ast}) \neq {s''}^{\ast}}}\mu_S(x, s)\big) = 0\)
and, thus, we deduce
\(\xi(s'',t'') = 0\).
\end{proof}

Now, considering Lemma~\ref{LemmaMain}, we proceed to the following theorem.

\begin{thm}\label{uu'}
Let $s,t\in S$.
If $s^{\ast}\neq {t}^{+}$, then we have $s\ast t=0$.
\end{thm}

Let $X_e=\{x_s\in X\mid s\in \widetilde{L}_e\widetilde{R}_e\}$.
Let $\theta_e(X_e)$ be the determinant of the submatrix $\widetilde{L}_e\times\widetilde{R}_e$ of the Cayley table $(Z(S),\boldsymbol{*})$, for every idempotent $e\in S$.

\begin{thm}\label{Main-Theorem}
Let $S\in \pv{ECOM}$. 
For $s \in S$, put $y_s=\sum\limits_{t\ll s}\mu_S(t, s)x_t$. 
Then, we have
\[
\theta_S(X)
= \pm\prod\limits_{e\in E(S)}\widetilde{\theta}_{e}(Y_e)
\]
where $Y_e=\{y_s\mid s \in \widetilde{L}_e\widetilde{R}_e\}$.
Moreover, the determinant of $S$ is non-zero if and only if $\widetilde{\theta}_{e}(Y_e)\neq 0$, for every idempotent $e$.
\end{thm}

\begin{proof}
Since $\mathbb{C}S\cong Z(\mathbb{C}S)$, by Theorem~\ref{Zisom}, we can take the set $$X=\{x_1,\ldots,x_{\abs{S}}\}$$ as the base of $Z(\mathbb{C}S)$. 
Let $M$ be a matrix that by rearranging and shifting the rows and columns of $C(X)$ so that the elements of the subset $\widetilde{L}_e$ being adjacent rows and the elements of the subset $\widetilde{R}_{e}$ being adjacent columns for every idempotent $e\in E(S)$. 
If $s\in \widetilde{L}_e$ and $t\in\widetilde{R}_{e'}$, for some distinct idempotents $e$ and $e'$, then we have $x_s\boldsymbol{*}x_t=0$.
Then, the determinant of the matrix $M$ is equal to $\prod\limits_{e\in E(S)}\widetilde{\theta}_{e}(X_e)$ and, thus, we have 
\[\theta_{(Z(\mathbb{C}S),X)}= \pm\prod\limits_{e\in E(S)}\widetilde{\theta}_{e}(X_e).\]

Let $P$ be the $S\times S$ matrix of $Z$ with respect to the bases of $\mathbb{C}S$ and $\mathbb{C}Z(S)$.
Since the relation $\ll$ is a partial order, the determinant of $P$ is equal to 1.
Now, we apply Theorem~\ref{ThetaCal}.
Since $Z$ is an isomorphism and, as discussed above, we have
\[
\theta_S(X)
= \widetilde{Z}^{-1}(\theta_{(\widetilde{Z}(\mathbb{C}S),X)})
= \widetilde{Z}^{-1}(\pm\prod\limits_{e\in E(S)}\widetilde{\theta}_{e}(X_e))
= \pm\prod\limits_{e\in E(S)}\widetilde{Z}^{-1}(\widetilde{\theta}_{e}(X_e)).
\]
By applying M\"{o}bius inversion, we have $Z(y_s)=x_s$. Therefore, we get that
\[
\theta_S(X)
= \pm\prod\limits_{e\in E(S)}\widetilde{\theta}_{e}(Y_e).
\]
\end{proof}

Note that in Theorem~\ref{Main-Theorem}, the sign of \(\prod\limits_{e\in E(S)}\widetilde{\theta}_{e}(Y_e)\) 
is contingent on whether the number of rearrangements and shifts applied to the rows and columns of $C(X)$ in order to construct the matrix \(M\)  is odd or even.
Example~\ref{Exa13} utilizes Theorem~\ref{Main-Theorem} to demonstrate that the determinant of the semigroup $S_9$ is non-zero.

%%%%%%%%%%%%%%%%%%%%%%%%%%%%%%%%%%%%%%%%%%%%%%%%%%%%%%%%%%%%%%%%%%%%%%%%%%%%%%%%%%%%%%%%%%%%%%%%%%%%%%%%%%%%%%%%%%%%%%%%%%%%%%%%%%%%%%%%%%%%%%
%%%%%%%%%%%%%%%%%%%%%%%%%%%%%%%%%%%%%%%%%%%%%%%%%%%%%%%%%%%%%%%%%%%%%%%%%%%%%%%%%%%%%%%%%%%%%%%%%%%%%%%%%%%%%%%%%%%%%%%%%%%%%%%%%%%%%%%%%%%%%%

\section{Appendix A}

In Appendix, we provide examples of the semigroups that the paper discusses.
We present detailed information for each semigroup, including its Cayley table and contracted semigroup determinants.
In the examples provided, we exclude the element zero from the rows and columns of the Cayley table.
If the multiplication of two non-zero elements in the Cayley table results in zero, we represent it with a dot.
%Additionally, we note the contracted semigroup determinants.

\begin{example}\label{Exa1}
\ 
\begin{center}%[n= 5]: (42, 2);
\begin{tabular}{ c | c c c c } 
\textbf{$S_1$} & $\cc{y}$ & $\cc{z}$ & $\cc{u}$ & $\cc{t}$\\ \hline 
$\cc{y}$  & . & . & . & $y$ \\ 
$\cc{z}$  & . & . & $z$  & $z$ \\ 
$\cc{u}$  & $y$  & . & $u$  & $u$ \\ 
$\cc{t}$  & $y$  & $z$  & $u$  & $t$ \\ 
\end{tabular}\\
\end{center}
$\widetilde{\theta}_{S_1}(X_{S_1})=-x^2_yx^2_z$.\\
%The partitions of $S_1$ respect the equivalence relations $\sim_r$, $\sim_l$ and $\sim$ are as follows:
%$$S_1/\sim_r=\{\{0\},\{y,t\},\{z,u\}\}, S_1/\sim_l=\{\{0\},\{y,u\},\{z,t\}\}\ \text{and}$$ 
%$$S_1/\sim=\{\{0\},\{y\},\{z\},\{u\},\{t\}\}.$$\\
We have $y^{\ast}=t$ and $y^{+}=u$.
\end{example}
%%%%%%%%%%%%%%%%%%%%%%%%%%%%%%%%%%%%%%%%%%%%%%%%%%%%%%%%%%%%%%%%%%%%%%%%%%%%%%%%%%%%%%%%%%%%%%%%%%%%%%

\begin{example}\label{Exa2}
\ 
\begin{center} %[n= 8]: (738, 1);
\begin{tabular}{ c | c c c c c c c } 
\textbf{$S_2$} & $\cc{y}$ & $\cc{z}$ & $\cc{u}$ & $\cc{t}$ & $\cc{w}$ & $\cc{v}$ & $\cc{q}$\\ \hline 
$\cc{y}$  & . & . & . & . & . & $y$  & $y$ \\ 
$\cc{z}$  & . & . & . & . & . & . & $z$ \\ 
$\cc{u}$  & . & . & . & $y$  & $y$  & $y$  & $u$ \\ 
$\cc{t}$  & . & . & $z$  & . & $z$  & $t$  & $t$ \\ 
$\cc{w}$  & . & . & $z$  & $y$  & $y$  & $t$  & $w$ \\ 
$\cc{v}$  & $y$  & . & $u$  & $y$  & $u$  & $v$  & $v$ \\ 
$\cc{q}$  & $y$  & $z$  & $u$  & $t$  & $w$  & $v$  & $q$ \\ 
\end{tabular}
\end{center}
$\widetilde{\theta}_{S_2}(X_{S_2})=x^3_yx^4_z$.\\
We have $t \ll w$, $u \ll w$ and $z=tu \not\ll ww = y$.
\end{example}

\begin{example}\label{Exa3}
\ 
\begin{center} %[n= 9]: (1, 1);
\begin{tabular}{ c | c c c c c c c c c} 
\textbf{$S_3$} & $\cc{a}$ & $\cc{b}$ & $\cc{s'}$ & $\cc{t'}$ & $\cc{s}$ & $\cc{t}$ & $\cc{e}$ & $\cc{f}$\\ \hline 
$\cc{a}$ & . & . & . & . & . & . & $a$  & $a$ \\ 
$\cc{b}$ & . & . & . & . & . & . & $b$  & $b$ \\ 
$\cc{s'}$ & . & . & . & $a$  & . & $a$  & $s'$  & $s'$ \\ 
$\cc{t'}$ & . & . & . & . & . & . & $t'$  & $t'$ \\ 
$\cc{s}$ & . & . & . & $a$  & . & $b$  & $s'$  & $s$ \\ 
$\cc{t}$ & . & . & . & . & . & . & $t$  & $t$ \\ 
$\cc{e}$ & $a$  & $b$  & $s'$  & $t'$  & $s$  & $t'$  & $e$  & $e$ \\ 
$\cc{f}$ & $a$  & $b$  & $s'$  & $t'$  & $s$  & $t$  & $e$  & $f$ \\ 
\end{tabular}
\end{center}
$\theta_{S_3}(X_{S_3})=0$.\\
We have $s'\ll s$, $t'\ll t$, $s't'=a$, $st=b$ and $a\not\ll b$. %Note that $s\boldsymbol{*} t\neq 0$.
\end{example}

\begin{example}\label{Exa4}
In the semigroup $S_1$ (Example~\ref{Exa1}),
we have $I^{\ast}_t = \{0, z, u\}$ and $I^{\ast}_t$ is not an ideal.
\end{example}

\begin{example}\label{Exa5}
\ 
\begin{center} %[n= 7]: (4325, 2);
\begin{tabular}{ c | c c c c c c } 
\textbf{$S_4$} & $\cc{y}$ & $\cc{z}$ & $\cc{u}$ & $\cc{t}$ & $\cc{w}$ & $\cc{v}$\\ \hline 
$\cc{y}$  & . & . & . & . & . & $y$ \\ 
$\cc{z}$  & . & . & . & . & $z$  & .\\ 
$\cc{u}$  & . & . & . & $u$  & $u$  & $u$ \\ 
$\cc{t}$  & . & $z$  & . & $t$  & $t$  & $t$ \\ 
$\cc{w}$  & $y$  & $z$  & . & $t$  & $w$  & $t$ \\ 
$\cc{v}$  & . & $z$  & $u$  & $t$  & $t$  & $v$ \\ 
\end{tabular}
\end{center}
$\widetilde{\theta}_{S_4}(X_{S_4})=-x^2_yx^2_zx^2_u$.\\
We have $wS_4w = \{0, z, t, w\}$, $vS_4v = \{0, u, t, v\}$, $I^{\ast}_w = \{0, u, t\}$, $I^{+}_v = \{0, z, t\}$, $I^{\ast}_w\not\subset wS_4w$ and $I^{+}_v\not\subset vS_4v$.
\end{example}

\begin{example}\label{Exa8}
Let $S_5$ be a semigroup defined as follows.
\begin{align*}
S_5=\langle 0, s, t, e, f, g, h, 1 \mid\ & se=s, ft=t, e^2=e, f^2=f, g^2=g, h^2=h,\\
                                       \ & ef=fe, gh=hg=0, ge=eg=g, gf=fg=g,\\
                                       \ & he=eh=h, hf=fh=h, sgt=sht=st,\\
                                       \ &0\ \text{is the zero of}\ S_5\ \text{and}\ 1\ \text{is the identity of}\ S_5 \rangle,
\end{align*}
$\theta_{S_5}(X_{S_5})=0$.\\
We have $s^{\ast}=e$, $t^{+}=f$, $E^{st}=\{g,h,ef\}$ and $g,h$ are minimal in $E^{st}$. 
Also, we have 
\[(st)^{\ast}=t^{\ast}=(et)^{\ast}=(ht)^{\ast}=(gt)^{\ast}=1\] 
and 
\[(st)^{+}=s^{+}=(sf)^{+}=(sh)^{+}=(sg)^{+}=1.\]
\end{example}

\begin{example}\label{Exa9}
\ 
\begin{center} %[n= 8]: (4873, 1);
\begin{tabular}{ c | c c c c c c c } 
\textbf{$S_6$} & $\cc{y}$ & $\cc{z}$ & $\cc{u}$ & $\cc{t}$ & $\cc{w}$ & $\cc{v}$ & $\cc{q}$\\ \hline 
$\cc{y}$ & . & . & $z$  & $z$  & $z$  & $y$  & $y$ \\ 
$\cc{z}$ & . & . & . & . & . & $z$  & $z$ \\ 
$\cc{u}$ & $z$  & . & $y$  & $y$  & $y$  & $u$  & $u$ \\ 
$\cc{t}$ & $z$  & . & $y$  & $y$  & $y$  & $u$  & $t$ \\ 
$\cc{w}$ & $z$  & . & $y$  & $y$  & $y$  & $w$  & $w$ \\ 
$\cc{v}$ & $y$  & $z$  & $u$  & $t$  & $u$  & $v$  & $v$ \\ 
$\cc{q}$ & $y$  & $z$  & $u$  & $t$  & $w$  & $v$  & $q$ \\ 
\end{tabular}
\end{center}
$\theta_{S_6}(X_{S_6})=0$.\\
We have $(u \sharp t) \sharp w \neq u \sharp (t \sharp w)$, since $u \sharp t = 0$ and $ u \sharp (t \sharp w)=utw$.
\end{example}

\begin{example}\label{Exa10}
%input{#S10.txt}
Let $S_7$ be the semigroup generated with the letters \[0, s,s_1,s_2,t,t_1,t_2,s_3,t_3,s_4,s_5,\] 
defined as follows
\begin{align*}
\langle&
xy = yx\ \text{and}\ x^2 = x,\ \text{for all}\ x,y\in \{s_1,s_2,t_1,t_2,s_3,t_3,s_4,s_5\},\\
\ &
s_3 < s_1, t_3 < t_2, s_4 < s_5 < s_2, s_5 < t_1, \\
\ &
s_1s=s, ss_2=s, t_1t=t, tt_2=t,
s_3ss_5=s_3s, s_5tt_3=tt_3,\\
\ &
\{s_2,s_4,s_5,t_1,t_2,t_3\}s,
s\{s_1,s_3,t_2,t_3\},
\{s_1,s_3,t_2,t_3\}t,
t\{s_1,s_2,s_3,s_4,s_5,t_1\}=0,\\
\ &
ss, ts, tt = 0, \ \text{and} \
0\ \text{is the zero of}\ S_7 \rangle.
\end{align*}
With the help of GAP~\cite{GAP4}, the semigroup \(S_7\) has been identified, and the following conditions have been verified.
\begin{align*}
\ &
s_1 = s^{+},
s_2 = s^{\ast}, 
t_1 = t^{+}, 
t_2 = t^{\ast},\\
\ & 
s_3 = {s'}^{+}, 
t_3 = {t'}^{\ast}, 
s_4 = {s'}^{\ast} = {t'}^{+},\ \text{where}\ s' = s_3ss_4\ \text{and}\ t' = s_4tt_3,\\
\ & 
s_5 = ({s'}^{+}s)^{\ast} = (t{t'}^{\ast})^{\ast}.
\end{align*}
Then, we have \(s' \ll s\) and \(t' \ll t\).
We deduce \({s'}\sharp {t'} \neq 0\), \(\varepsilon({s'}^{+}s,t{t'}^{\ast})=({s'}^{+}s)^{\ast}\), and \(\varepsilon(s',t') = {s'}^{\ast}\) 
with \(({s'}^{+}s)^{\ast} \neq {s'}^{\ast}\). 
\end{example}

\begin{example}\label{Exa11}
\ 
\begin{center} %[n= 8]: (278360, 1)
\begin{tabular}{ c | c c c c c c c } 
\textbf{\cc{$S_{10}$}} & $\cc{y}$ & $\cc{z}$ & $\cc{u}$ & $\cc{t}$ & $\cc{w}$ & $\cc{v}$ & $\cc{q}$\\ \hline 
$\cc{y}$  & . & . & . & . & . & $y$  & $y$ \\ 
$\cc{z}$  & . & . & . & . & $z$  & $z$  & $z$ \\ 
$\cc{u}$  & . & . & . & . & $z$  & $u$  & $u$ \\ 
$\cc{t}$  & . & . & . & . & $z$  & $u$  & $t$ \\ 
$\cc{w}$  & $y$  & . & $y$  & $y$  & $w$  & $w$  & $w$ \\ 
$\cc{v}$  & $y$  & $z$  & $u$  & $u$  & $w$  & $v$  & $v$ \\ 
$\cc{q}$  & $y$  & $z$  & $u$  & $t$  & $w$  & $v$  & $q$ \\ 
\end{tabular} 
\end{center}
$\theta_{S_8}(X_{S_8})=0$.\\
We have 
\(z \leq u\), \(y \leq t\)
and
\[\varepsilon(z^{+}z, yy^{\ast}) (=z^{\ast}= y^{+} = w)\neq \varepsilon(z^+u, ty^{\ast}) (=v).\]
Also, we have \(z \sharp y = 0\).
\end{example}

\begin{example}\label{Exa12}
\ 
\begin{center} %[n= 8]: (278360, 737);
\begin{tabular}{ c | c c c c c c c } 
\textbf{$S_9$} & $\cc{y}$ & $\cc{z}$ & $\cc{u}$ & $\cc{t}$ & $\cc{w}$ & $\cc{v}$ & $\cc{q}$\\ \hline 
$\cc{y}$  & . & . & . & . & . & $y$  & $y$ \\ 
$\cc{z}$  & . & . & . & . & . & . & $z$ \\ 
$\cc{u}$  & . & . & . & $y$  & $y$  & . & $u$ \\ 
$\cc{t}$  & . & . & $z$  & . & $z$  & $t$  & $t$ \\ 
$\cc{w}$  & . & . & $z$  & $y$  & $y$  & $t$  & $w$ \\ 
$\cc{v}$  & $y$  & . & $u$  & . & $u$  & $v$  & $v$ \\ 
$\cc{q}$  & $y$  & $z$  & $u$  & $t$  & $w$  & $v$  & $q$ \\ 
\end{tabular}
\end{center}
The relations \(t \ll w\) and \(u \ll w\) are irreducible, \(t^{\ast} = u^{+} (= v)\) and \(w^{\ast} = w^{+} (= q)\).
We have \(\varepsilon(t^{+}w,wu^{\ast}) (= q) \neq t^{\ast}\).
\end{example}

\begin{example}\label{Exa13}
\ 
Consider the semigroup $S_9$ in Example~\ref{Exa12}.\\
The table of the semigroup $Z(S_9)$ %, using the modification described at the beginning of Appendix 
with the multiplication $\boldsymbol{*}$ is as follows:\\
\begin{center}
\begin{tabular}{ c | c c c c c c c }  
\textbf{$(Z(S_9),\boldsymbol{*})$} & $\cc{y}$ & $\cc{z}$ & $\cc{u}$ & $\cc{t}$ & $\cc{u+t+w}$ & $\cc{v}$ & $\cc{v+q}$\\ \hline 
$\cc{y}$ & . & . & . & . & . & $y$  & $y$ \\ 
$\cc{z}$ & . & . & . & . & . & . & $z$ \\ 
$\cc{u}$ & . & . & . & $y$  & $y$  & . & $u$ \\ 
$\cc{t}$ & . & . & $z$  & . & $z$  & $t$  & $t$ \\ 
$\cc{u+t+w}$ & . & . & $z$  & $y$  & $y$  & $t$  & $u+t+w$ \\ 
$\cc{v}$ & $y$  & . & $u$  & . & $u$  & $v$  & $v$ \\ 
$\cc{v+q}$ & $y$  & $z$  & $u$  & $t$  & $u+t+w$  & $v$  & $v+q$  
\end{tabular} 
\end{center}
By Theorem~\ref{Zisom}, the table on the left side can be used for $Z(S_9)$. 
\begin{center}
\begin{tabular}{ll } 
\begin{tabular}{ c | c c c c c c c } 
\textbf{$(Z(S_9),\boldsymbol{*})$} & $\cc{y}$ & $\cc{z}$ & $\cc{u}$ & $\cc{t}$ & $\cc{w}$ & $\cc{v}$ & $\cc{q}$\\ \hline 
$\cc{y}$ & .    & .    & .    & .    & .    & $y$  & .\\
$\cc{z}$ & .    & .    & .    & .    & .    & .    & $z$  \\ 
$\cc{u}$ & .    & .    & .    & $y$  & .    & .    & $u$  \\ 
$\cc{t}$ & .    & .    & $z$  & .    & .    & $t$  & .\\ 
$\cc{w}$ & .    & .    & .    & .    & $-z$ & .    & $w$  \\ 
$\cc{v}$ & $y$  & .    & $u$  & .    & .    & $v$  & .\\ 
$\cc{q}$ & .    & $z$  & .    & $t$  & $w$  & .    & $q$   
\end{tabular} &
\begin{tabular}{ c | c c c c c c c } 
\textbf{$M$} & $\cc{y}$ & $\cc{u}$ & $\cc{v}$ & $\cc{z}$ & $\cc{t}$ & $\cc{w}$ & $\cc{q}$\\ \hline 
$\cc{y}$ & . & . & $y$  & . & . & . & .\\ 
$\cc{t}$ & . & $z$  & $t$  & . & . & . & .\\ 
$\cc{v}$ & $y$  & $u$  & $v$  & . & . & . & .\\ 
$\cc{z}$ & . & . & . & . & . & . & $z$ \\ 
$\cc{u}$ & . & . & . & . & $y$  & . & $u$ \\ 
$\cc{w}$ & . & . & . & . & . & $-z$  & $w$ \\ 
$\cc{q}$ & . & . & . & $z$  & $t$  & $w$  & $q$ 
\end{tabular}
\end{tabular}
\end{center}
The table $M$ is on the right side that by rearranging and shifting the rows and columns of $Z(S_9)$ so that the elements of the subset $\widetilde{L}_e$ being adjacent rows and the elements of the subset $\widetilde{R}_{e}$ being adjacent columns for every idempotent $e\in E(S_9)$. 
Then, it is easy to compute that the determinant of the matrix $M$ is non-zero and equal to $-x_y^3x_z^4$. 
Consequently, the determinant of $\widetilde{\theta}_{S_9}(X_{S_9})$ is also non-zero, with a value of $x_y^3x_z^4$. 
\end{example}

%%%%%%%%%%%%%%%%%%%%%%%%%%%%%%%%%%%%%%%%%%%%%%%%%%%%%%%%%%%%%%%%%%%%%%%%%%%%%%%%%%%%%%%%%%%%%%%%%%%%%%%%%%%%%%%%%%%%%%%%%%%%%%%%%%%%%%%%%%%%%%
%%%%%%%%%%%%%%%%%%%%%%%%%%%%%%%%%%%%%%%%%%%%%%%%%%%%%%%%%%%%%%%%%%%%%%%%%%%%%%%%%%%%%%%%%%%%%%%%%%%%%%%%%%%%%%%%%%%%%%%%%%%%%%%%%%%%%%%%%%%%%%

\section*{Acknowledgments}
The author was partially supported by CMUP, which is financed by
national funds through FCT -- Fundação para a Ciência e a Tecnologia,
I.P., under the project UIDB/00144/2020. The author also
acknowledges FCT support through a contract based on the “Lei do
Emprego Científico” (DL 57/2016).

\bibliographystyle{plain}
\bibliography{ref-Det}

\begin{thebibliography}{10}

\bibitem{Alm}
J.~Almeida.
\newblock {\em Finite semigroups and universal algebra}, volume~3 of {\em
  Series in Algebra}.
\newblock World Scientific Publishing Co., Inc., River Edge, NJ, 1994.
\newblock Translated from the 1992 Portuguese original and revised by the
  author.

\bibitem{Ash}
C.~J. Ash.
\newblock Finite semigroups with commuting idempotents.
\newblock {\em J. Austral. Math. Soc. Ser. A}, 43(1):81--90, 1987.

\bibitem{Assem-Ibrahim}
I.~Assem, D.~Simson, and A.~Skowro\'{n}ski.
\newblock {\em Elements of the representation theory of associative algebras.
  {V}ol. 1}, volume~65 of {\em London Mathematical Society Student Texts}.
\newblock Cambridge University Press, Cambridge, 2006.
\newblock Techniques of representation theory.

\bibitem{Benson}
D.~J. Benson.
\newblock {\em Representations and cohomology. {I}}, volume~30 of {\em
  Cambridge Studies in Advanced Mathematics}.
\newblock Cambridge University Press, Cambridge, second edition, 1998.
\newblock Basic representation theory of finite groups and associative
  algebras.

\bibitem{Cli-Pre}
A.~H. Clifford and G.~B. Preston.
\newblock {\em The algebraic theory of semigroups. {V}ol. {I}}.
\newblock Mathematical Surveys, No. 7. American Mathematical Society,
  Providence, R.I., 1961.

\bibitem{Fou-Gom-Gou}
J.~Fountain, G.M.S. Gomes, and V.~Gould.
\newblock Enlargements, semiabundancy and unipotent monoids.
\newblock {\em Comm. Algebra}, 27(2):595--614, 1999.

\bibitem{Frobenius1903theorie}
F.~G. Frobenius.
\newblock Theorie der hyperkomplexen gr{\"o}{\ss}en.
\newblock {\em Sitzungsber. Press. Akad. Wissen., Leipzig}, 24:504--537, 1903.

\bibitem{GAP4}
The GAP~Group.
\newblock {\em {GAP -- Groups, Algorithms, and Programming, Version 4.12.2}},
  2022.

\bibitem{Gre}
J.~A. Green.
\newblock On the structure of semigroups.
\newblock {\em Ann. of Math. (2)}, 54:163--172, 1951.

\bibitem{Lawson-Sem-Cat}
M.~V. Lawson.
\newblock Semigroups and ordered categories. {I}. {T}he reduced case.
\newblock {\em J. Algebra}, 141(2):422--462, 1991.

\bibitem{Lindstr}
B.~Lindstr\"{o}m.
\newblock Determinants on semilattices.
\newblock {\em Proc. Amer. Math. Soc.}, 20:207--208, 1969.

\bibitem{Okn}
J.~Okni{\'n}ski.
\newblock {\em Semigroup algebras}, volume 138 of {\em Monographs and Textbooks
  in Pure and Applied Mathematics}.
\newblock Marcel Dekker, Inc., New York, 1991.

\bibitem{Ponizovski}
I.~S. Ponizovski\u{\i}.
\newblock The {F}robeniusness of the semigroup algebra of a finite commutative
  semigroup.
\newblock {\em Izv. Akad. Nauk SSSR Ser. Mat.}, 32:820--836, 1968.

\bibitem{Rho-Ste}
J.~Rhodes and B.~Steinberg.
\newblock {\em The {$q$}-theory of finite semigroups}.
\newblock Springer Monographs in Mathematics. Springer, New York, 2009.

\bibitem{Smith}
H.~J.~S. Smith.
\newblock On the {V}alue of a {C}ertain {A}rithmetical {D}eterminant.
\newblock {\em Proc. Lond. Math. Soc.}, 7:208--212, 1875/76.

\bibitem{IncidenceAlgebras}
E.~Spiegel and Christopher~J. O'Donnell.
\newblock {\em Incidence algebras}, volume 206 of {\em Monographs and Textbooks
  in Pure and Applied Mathematics}.
\newblock Marcel Dekker, Inc., New York, 1997.

\bibitem{Ste-Fac-det}
B.~Steinberg.
\newblock Factoring the {D}edekind-{F}robenius determinant of a semigroup.
\newblock {\em J. Algebra}, 605:1--36, 2022.

\bibitem{Wenger}
R.~Wenger.
\newblock Some semigroups having quasi-{F}robenius algebras. {II}.
\newblock {\em Canadian J. Math.}, 21:615--624, 1969.

\bibitem{Wilf}
H.~S. Wilf.
\newblock Hadamard determinants, {M}\"{o}bius functions, and the chromatic
  number of a graph.
\newblock {\em Bull. Amer. Math. Soc.}, 74:960--964, 1968.

\bibitem{Wood-Duality}
J.~A. Wood.
\newblock Duality for modules over finite rings and applications to coding
  theory.
\newblock {\em Amer. J. Math.}, 121(3):555--575, 1999.

\bibitem{Wood}
J.~A. Wood.
\newblock Factoring the semigroup determinant of a finite commutative chain
  ring.
\newblock In {\em Coding theory, cryptography and related areas ({G}uanajuato,
  1998)}, pages 249--259. Springer, Berlin, 2000.

\end{thebibliography}

\end{document}